\documentclass[a4paper,12pt]{article}
\usepackage{amsfonts,amssymb}
\usepackage{theorem}
\input diagrams
{\theorembodyfont{\rm}
  \newtheorem{defi}{Definition}[section]
  \newtheorem{rem}[defi]{Remark}
  \newtheorem{exa}[defi]{Example}
  \newtheorem{exas}[defi]{Examples}}
  
  \newtheorem{prop}[defi]{Proposition}
  \newtheorem{thm}[defi]{Theorem}
  
\newcommand{\PP}{{\mathbb P}}

\newcommand{\cC}{{\mathcal C}}
\newcommand{\cD}{{\mathcal D}}
\newcommand{\cG}{{\mathcal G}}
\newcommand{\cR}{{\mathcal R}}
\newcommand{\cS}{{\mathcal S}}
\newcommand{\cT}{{\mathcal T}}

\newcommand{\fC}{{\mathfrak C}}

\newcommand{\Aut}{{\mathrm{Aut}}}

\newcommand{\End}{{\mathrm{End}}}

\newcommand{\Char}{{\mathrm{char}}}
\newcommand{\T}{{\mathrm{T}}}

\newcommand{\id}{{\mathrm{id}}}
\newcommand{\dis}
                 {{\mathrel{\scriptstyle{\triangle}}}}

\newcommand{\eps}{{\varepsilon}}
\newcommand{\GF}{{\mathrm{GF}}}
\newcommand{\GL}{{\mathrm{GL}}}

\newcommand{\M}{{\mathrm{M}}}
\let\phi=\varphi

\let\theta=\vartheta

\newcommand{\DelimArray}[4]{\left#1\begin{array}{*{#3}{r}}#4\end{array}\right#2}
\newcommand{\SDelimArray}[4]{\hbox{\scriptsize\arraycolsep=.5\arraycolsep
  $\left#1\!\!\begin{array}{*{#3}{r}}#4\end{array}\!\!\right#2$}}

\newcommand{\Mat}{\DelimArray()}
\newcommand{\SMat}{\SDelimArray()}

\newenvironment{proof}
    {\begin{trivlist} \item {\sl Proof:}} 
    {\/ $\square$ \end{trivlist}}

\date{\normalsize
{\em Dedicated to Walter Benz on the occasion of his 70th birthday.}
}
\begin{document}
\title{Projective Representations\\
\vspace{.3cm} {II. Generalized chain geometries}}
\author{Andrea Blunck\thanks{Supported by a Lise Meitner
 Research Fellowship
of the Austrian Science Fund (FWF), projects M529-MAT, M574-MAT.} \and {Hans
Havlicek}}

\maketitle

\begin{abstract} \noindent
In this paper, projective representations of generalized chain geometries are
investigated, using the concepts and results of \cite{blu+h-00a}. In
particular, we study under which conditions such a projective representation
maps the chains of a generalized chain geometry $\Sigma(F,R)$ to reguli; this
mainly depends on how the field $F$ is embedded in the ring $R$. Moreover, we
determine all bijective morphisms of a certain class of generalized chain
geometries with the help of projective representations.

\end{abstract}
\parskip .8mm
\parindent0cm

\section{Introduction}
In  \cite{blu+h-00a}, which is Part I of this publication,  we considered the
projective line $\PP(R)$ over a ring $R$ with $1$ and constructed projective
representations of $\PP(R)$ with the help of unitary $(K,R)$-bimodules, where
$K$ is a not necessarily commutative field. Such a  $(K,R)$-bimodule $U$ gives
rise to a projective representation $\Phi$ that maps $\PP(R)$ into the set
$\cG$ of those subspaces of the projective space $\PP(K,U\times U)$ that are
isomorphic to one of their complements.

If $R$ contains a subfield  $F$ with $1\in F$, then the $F$-sublines turn
$\PP(R)$ into a {\em generalized chain geometry} $\Sigma(F,R)$, compare
\cite{blu+h-99a}. Ordinary {\em chain geometries}, where $R$ is an $F$-algebra,
are studied, e.g., in \cite{benz-73} and \cite{herz-95}.

We want to investigate  the images of the chains of a generalized chain
geometry $\Sigma(F,R)$ under projective representations of $\PP(R)$.
It is well known from
\cite{herz-95}, Section 4,  that chains appear as reguli if $R$ is a
finite-dimensional $F$-algebra, $F=K$,  and the representation $\Phi$ is
injective. In general, this is no longer true. In view of the various examples
in the present paper,  a unified geometric description of the $\Phi$-images of
chains seems very difficult. Hence we focus our attention to those cases where
the $\Phi$-images of chains are reguli in the sense of \cite{blunck-00a}, i.e.,
reguli in not necessarily pappian spaces of arbitrary dimension.

As the $\Phi$-images of any two chains are projectively equivalent, it suffices
to discuss the $\Phi$-image of the standard chain $\PP(F)\subset\PP(R)$. We
proceed in  two steps:

In Sections 2 and 3 we discuss arbitrary projective representations  of
$\PP(F)$ and characterize those  where $\PP(F)$ goes over to a regulus (Theorem
\ref{regulus}) or a ``direct sum'' of reguli (Theorem \ref{quasiReg}). Such
representations exist for all fields $F$.

Next, in Section 4, we take into account that $F\subset R$. Then a projective
representation $\Phi$ of $\PP(R)$ determines the projective representation
$\Phi|_{\PP(F)}$  of $\PP(F)$; it  depends essentially on how $F$ is embedded
in  $R$. There are cases where $\Sigma(F,R)$ does not admit any projective
representation mapping its chains to reguli (see Example \ref{noRep}).

Finally, in Section 5 we use projective representations in order to determine
all bijective morphisms between chain geometries over rings of $2\times
2$-matrices.

Throughout this paper we  adopt the notions of \cite{blu+h-00a}. All our rings
have a unit element which is preserved by homomorphisms, inherited by subrings,
and acts unitally on modules.

\section{Transversals}

Let $R$ be any ring. We recall from \cite{blu+h-00a}, Section 4, that a
projective representation of the projective line $\PP(R)$ is given by means of
a left vector space $U$ over a not necessarily commutative field $K$ and a ring
homomorphism $\phi:R\to\End_K(U)$, or, equivalently, by means of a
$(K,R)$-bimodule $U$. The projective representation obtained from these data is
called $\Phi$ and maps $\PP(R)$ into the set $\cG$ of those subspaces of the
projective space $\PP(K,U\times U)$ that are isomorphic to one of their
complements.

As mentioned before, in the classical case  chains go over to  reguli. Since
reguli have many transversals, we
start by defining transversals of an arbitrary subset $\cS$
of $\cG$.

\begin{defi}
Let  $T$ be a line in $\PP(K,U\times U)$ and let  $\cS$ be a subset of $\cG$.
Then $T$ is a
\begin{enumerate}
\item  {\em weak transversal} of $\cS$, if $T$ meets each element
of $\cS$ in a unique point.
\item {\em transversal} of $\cS$, if $T$ is a weak transversal and
each point
of $T$ lies on an element of~$\cS$.
\end{enumerate}
\end{defi}

We know that $\PP(R)^\Phi$ contains $U\times\{0\}$, $\{0\}\times U$, and
$\{(u,u)\mid u\in U\}$. So one directly sees that a weak transversal of
$\PP(R)^\Phi$ must have the form
\begin{equation} T=Ku\times Ku
\end{equation}
for a suitable $u\in U\setminus\{0\}$. This description shows in particular
that any two weak transversals of $\PP(R)^\Phi$ are skew. Moreover, if
$U=\{0\}$, then of course $\PP(R)^\Phi$ does not have any weak transversals;
the associated projective space is empty and hence contains no lines.

We can characterize the elements $u$ of $U$ that give rise to (weak)
transversals of $\PP(R)^\Phi$:

\begin{prop}\label{eigen}
For $u\in U\setminus\{0\}$ the following statements are equivalent:
\begin{enumerate}
\item  $T=Ku\times Ku$  is
a weak transversal of $\PP(R)^\Phi$.
\item $u$ is an eigenvector
of all $\rho_a:U\to U:u\mapsto u\cdot a$ ($a\in R$).
\item $Ku$ is a sub-bimodule of the $(K,R)$-bimodule $U$.
\end{enumerate}
In this case, $\alpha:R\to K$ with
$u^{\rho_a}=a^\alpha u$ is a homomorphism of rings, and $T$ is a
transversal of $\PP(R)^\Phi$ exactly if $\alpha$ is surjective,
or,
equivalently, if $Ku$ is a cyclic submodule of the right $R$-module $U$.
\end{prop}
\begin{proof}
(a) $\Rightarrow$ (b): The line $T$ meets $\{(u^{\rho_a},u)\mid u\in U
\}\in\PP(R)^\Phi$. So $(u^{\rho_a},u)=(ku,u)$ for some $k\in K$. (b)
$\Rightarrow$ (c) is clear. (c) $\Rightarrow$ (a): compare \cite{blu+h-00a},
Proposition 4.8. The rest is a straightforward calculation.
\end{proof}

Now we consider the special case that the ring $R=:F$ is a field. If in
addition $U\ne\{0\}$, then $\phi:F\to\End_K(U)$ is injective, i.e., $\Phi$ is
faithful, and $\Char(F)=\Char(K)$. Moreover,  the homomorphism $\alpha:F\to K$,
with $u^{\rho_x}=x^\alpha u$, associated to the eigenvector $u$ of all
$\rho_x$, $x\in F$,  is a monomorphism of fields. It is  an isomorphism of
fields, exactly if $T=Ku\times Ku$ is a transversal. This implies that
$\PP(F)^\Phi$ cannot have any weak transversals if there is no monomorphism
$\alpha:K\to F$, and $\PP(F)^\Phi$ cannot have any transversals if $F\not\cong
K$.

But also in the special case of $F=K$ there are examples where
$\PP(F)^\Phi=\PP(K)^\Phi$ does not have any weak transversals at all. It can
also occur that $\PP(F)^\Phi$ has no transversals but many weak transversals,
that $\PP(F)^\Phi$ has exactly one transversal, and  that $\PP(F)^\Phi$ has at
least one transversal:

\begin{exas}\label{exasTrsv}
\begin{enumerate}
\item
Let $L$ be a commutative field, $K$ a subfield of $L$ and $\alpha\in\Aut(L)$
such that $K^\alpha\not\subset K$. Let $U\ne\{0\}$  be a left vector space over
$L$. Then $U$ becomes a $(K,K)$-bimodule by setting $k\cdot u:=ku$, $u\cdot
k:=k^\alpha u$. Consider a $k\in K$ with $k^\alpha\not\in K$. For each $u\in
U\setminus\{0\}$ we have $u^{\rho_k}=k^\alpha u\not\in K u$, so $u$ is not an
eigenvector of $\rho_k$. This means that $\PP(K)^\Phi$ has no weak
transversals.
\item
Let $K$ be a commutative field and let $\alpha:K\to K$ be a non-surjective
monomorphism of fields. Let  $U$ be any left vector space over $K$. Then $U$
becomes a $(K,K)$-bimodule by setting  $u\cdot k:=k^\alpha u$. Consider an
arbitrary  $u\in U\setminus\{0\}$. For each $k\in K$ we see that  $u$ is an
eigenvector of  $\rho_k$, with associated monomorphism $\alpha$. So $T=Ku\times
Ku$ is a weak transversal of $\PP(K)^\Phi$, but no transversal because $\alpha$
is not surjective.
\item Let $K$ be a commutative field that is an inseparable quadratic
extension of a subfield~$L$. Then $K=L+Li$ with  $i^2\in L$, and $\Char(K)=2$.
One can easily check that $$\phi:K\to \M(2\times 2,K): a+bi\mapsto \Mat2{a+bi&0
\ \ \   \\b \ \ \ &a+bi}$$ is an injective homomorphism of rings and thus makes
the left vector space $U=K^2$ a faithful $(K,K)$-bimodule. Obviously, $u=(1,0)$
is an eigenvector of all $\rho_k$, $k\in K$, with associated automorphism
$\alpha=\id$. So $T=Ku\times Ku$ is a transversal of $\PP(K)^\Phi$. Moreover,
for $k\in K\setminus L$, there is no eigenvector of $\rho_k$ that does not
belong to $Ku= K(1,0) $. Hence $\PP(K)^\Phi$ does not have any other
transversals.
\item
Let  $K$ be a subfield of a ring  $R$. Then $U=R$ is a $(K,K)$-bimodule with
$k\cdot u:=ku$, $u\cdot k:=uk$. The weak transversals of $\PP(K)^\Phi$ are
exactly the lines $Ku\times  Ku$ with $u\in R\setminus\{0\}$ satisfying
$uK\subset Ku$. Such a line is a transversal if, and only if, $uK=Ku$. In
particular, $\PP(K)^\Phi$ has at least one transversal, namely, $T=K\times K$,
where $u=1$.
\end{enumerate}
\end{exas}

Let $F$ be a field, and let $U$ be a $(K,F)$-bimodule. We want to study the
transversals of the projective model $\PP(F)^\Phi$ more closely. Because of
Proposition \ref{eigen} the existence of a transversal implies that $F\cong K$.
So we now restrict ourselves to the case that $F=K$. Since
$\PP(F)^\Phi=\PP(K)^\Phi$ consists of pairwise complementary subspaces, each
point of a transversal of $\PP(K)^\Phi$ lies on exactly one element of
$\PP(K)^\Phi$. So for two transversals $T_1$, $T_2$ of $\PP(K)^\Phi$ it makes
sense to define the following mapping
\begin{equation}\label{stern}
 \pi_{12}:T_1\to T_2: T_1\cap p^\Phi\mapsto T_2\cap p^\Phi,
\end{equation}
 where $p\in\PP(K)$. We can also describe this mapping
algebraically: Let $T_i=Ku_i\times Ku_i$, with $\alpha_i\in\Aut(K)$ associated
to the eigenvector $u_i$ ($i=1,2$). Then a point
$K(x^{\alpha_1}u_1,y^{\alpha_1}u_1)$ of $T_1$ (with $x,y\in K$) has the
${\pi_{12}}$-image $K(x^{\alpha_2}u_2,y^{\alpha_2}u_2)$.

We call $T_1$ and $T_2$ {\em projectively linked}, if $\pi_{12}$ is a
projectivity. The relation ``projectively linked'' is an equivalence relation
on the set of all transversals of $\PP(K)^\Phi$. Obviously,  the transversals
$T_1$ and $T_2$ are projectively linked if, and only if,
$\alpha_1^{-1}\alpha_2$ is an inner automorphism of $K$.

\section{Reguli}

We now study the question under which conditions the image of
$\PP(F)$ under a projective representation $\Phi$ is a regulus. As
before, we restrict ourselves to the case $F=K$.  We refer to
\cite{blunck-00a} for a synthetic definition of a regulus in a not
necessarily pappian space of arbitrary dimension. However, this
definition is rather involved and for our purposes the description of
a regulus in formula (\ref{lambda3}) below will be sufficient. We
always assume that $U\ne\{0\}$. Hence by our global assumptions each
representation $K\to\End_K(U)$ is faithful. Note that for the trivial
case $U=\{0\}$ the theorems of this section are also true, if one
defines that a regulus in the empty projective space $\PP(K,U\times
U)$ consists exactly of the empty set.

 Let $U\ne\{0\}$ be a left vector space over $K$. The
ring $\End_K(U)$ contains many copies of $K$: Let $(b_i)_{i\in I}$ be a basis
of $U$. Then
\begin{equation}\label{lambda1}
  \lambda:K\to \End_K(U):k\mapsto k^\lambda,
\end{equation}
 where $ k^\lambda$ is the linear mapping given by
\begin{equation}\label{lambda2}
 b_i\mapsto kb_i,
\end{equation}
embeds $K$ into $\End_K(U)$. The projective representation $\Lambda:
\PP(K)\to\cG$ associated to $\lambda$ has the form
\begin{equation}\label{lambda3}
 K(k,l)\mapsto \Big\{ \Big( \sum_{i\in I} x_i k b_i , \sum_{i\in I} x_i l b_i \Big)
 \mid
 x_i\in K \Big\};
\end{equation}
cf.\ \cite{blu+h-00a}, Theorem 4.2.
 In \cite{blunck-00a}, Theorem 3.1, it is shown that
 the image of $\PP(K)$ under this projective representation
is a regulus. Further, from Proposition \ref{eigen}, if $u=\sum_{i\in I} z_i
b_i\neq 0$ with $z_i$ in the centre of $K$, then $T=Ku\times Ku$ is a
transversal of the regulus $\PP(K)^\Lambda$. Conversely, each weak transversal
of $\PP(K)^\Lambda$ has this form, whence it is a transversal.

\begin{thm}\label{regulus}
Let $\phi:K\to\End_K(U)$ be a representation and let $\Phi$ be the associated
projective representation  of the projective line over $K$ in the projective
space $\PP(K,U\times U)$. Then the following are equivalent:
\begin{enumerate}
 \item $\PP(K)^\Phi$ is a regulus.
 \item Any two transversals of $\PP(K)^\Phi$ are projectively
 linked and all transversals of $\PP(K)^\Phi$ generate
 $\PP(K,U\times U)$.
 \item There exists a basis $(b_i)_{i\in I}$ of $U$ and an
 $\alpha\in\Aut(K)$ such that for all $k\in K$ the mapping
  $\rho_k=k^\phi\in\End_K(U)$ is given
 by $b_i\mapsto k^\alpha b_i$ ($i\in I$).
\end{enumerate}
\end{thm}
\begin{proof}
(a) $\Rightarrow$ (b): By \cite{blunck-00a}, Lemma 2.8, the  regulus
$\PP(K)^\Phi$ is projectively equivalent to the regulus $\PP(K)^\Lambda$, where
$\lambda$ is given by (\ref{lambda1}) and (\ref{lambda2}). So it suffices to
prove the assertion for $\PP(K)^\Lambda$. Let $T=Ku\times Ku$ be a transversal
of $\PP(K)^\Lambda$, i.e., $u$ is an eigenvector of all mappings
(\ref{lambda2}) with $k\in K$. Put $u=\sum_{i\in I} x_ib_i$, where without loss
of generality one coordinate, say $x_j$, equals $1$. We read off from the
$j$-th coordinate that $u$ goes over to $ku$ under each mapping
(\ref{lambda2}). Since $T$ has been chosen arbitrarily, this means that any two
transversals are projectively linked. Obviously, the transversals $Kb_i\times
Kb_i$ generate the entire space.

 (b) $\Rightarrow$ (c): Let $(T_i)_{i\in I} $
be a minimal family of transversals generating $\PP(K,U\times U)$. Each $T_i$
can be written as $Ku_i\times Ku_i$, where $(u_i)_{i\in I}$ is a basis of $U$,
and there are automorphisms $\alpha_i\in \Aut(K)$ with $u_i\cdot
k=k^{\alpha_i}u_i$. We fix one $j\in I$. As any two transversals are
projectively linked, each product $\alpha_j^{-1}\alpha_i$ is an inner
automorphism, say $x\mapsto k_i^{-1}xk_i$ with $k_i\in K^*$. Now
\begin{equation}\label{formel}
b_i:=k_iu_i,\quad\alpha:=\alpha_j
\end{equation}
have the required properties.

(c) $\Rightarrow$ (a): We observe that the representation $\lambda$ given by
(\ref{lambda1}) and (\ref{lambda2}) and the representation $\phi$ satisfy
$\phi=\alpha\lambda$. Since $\alpha:K\to K$ is a bijection, we have that
$\PP(K)^\Phi=\PP(K)^\Lambda$. As has been remarked before, this is a regulus.
\end{proof}

\begin{rem}
The existence of a basis $(b_i)_{i\in I}$ of $U$ and of automorphisms
$\alpha_i\in\Aut(K)$, such that $u_i\cdot k=k^{\alpha_i}u_i$ holds true for all
$i\in I$ and all $k\in K$, is equivalent to the existence of a family of
transversals of $\PP(K)^\Phi$ which generates the entire space. Let
$(T_i)_{i\in I}$ be a minimal generating family of transversals. If we fix one
$j\in I$, then $$\PP(K)^\Phi=\left\{\bigoplus_{i\in I} p^{\pi_{ji}}\mid p\in
T_j\right\},$$ where the bijections $\pi_{ji}$ are defined according to
(\ref{stern}).
\end{rem}

Note that if the transversals span only a subspace of the entire projective
space, then  similar statements hold true for the trace of $\PP(K)^\Phi$ in
this subspace, since this subspace corresponds to a sub-bimodule of the
$(K,K)$-bimodule $U$ (compare \cite{blu+h-00a}, Proposition 4.8).

\begin{thm}\label{quasiReg}
Assume that $\PP(K,U\times U)$ is spanned by the transversals of $\PP(K)^\Phi$.
Let $\cT_\theta$, $\theta\in\Theta$, be the equivalence classes  of
projectively linked transversals, and put $U_\theta\times U_\theta$ for the
subspace generated by $\cT_\theta$. Then
\begin{equation}\label{direct}
U\times U=\bigoplus_{\theta\in\Theta}(U_\theta\times U_\theta).
\end{equation}
 For each
$\theta\in\Theta$, the trace of $\PP(K)^\Phi$ in $\PP(K,U_\theta\times
U_\theta)$ is a regulus.
\end{thm}
\begin{proof}
Let $(T_i)_{i\in I}$ be a minimal family of transversals generating the entire
space. Then there is a basis $(u_i)_{i\in I}$ of $U$ such that $T_i=Ku_i\times
Ku_i$ holds for all $i\in I$.  Now suppose that $T=Ku\times Ku$ ($u\in
U\setminus \{0\}$) is any transversal. Then $T\in\cT_\theta$ for some
$\theta\in\Theta$. There exists a finite subset $I_u\subset I$ such that
$u=\sum_{i\in I_u} x_iu_i$ with $x_i\ne 0$ for all $i\in I_u$. For each $k\in
K$, the vectors $x_iu_i$ ($i\in I_u$) and $u$ are eigenvectors of
$\rho_k\in\End_K(U)$. By calculating $u^{\rho_k}$ in two ways, it follows that
these eigenvectors belong to the same eigenvalue, say $k^\alpha$ with
$\alpha\in\Aut(K)$. Hence the transversals $T_i$ ($i\in I_u$) and $T$ are
projectively linked, so that all $T_i$ ($i\in I_u$)  are in class $\cT_\theta$.
This implies that the subspace $U_\theta\times U_\theta$ is spanned by the set
$\{T_i\in\cT_\theta\mid i\in I\}$. Now (\ref{direct}) is obvious, and the
remaining assertion follows from Theorem~\ref{regulus}.
\end{proof}

In the finite-dimensional case  we have, by Theorem \ref{regulus}, that
$\PP(K)^\Phi$ is a regulus if, and only if,  the elements of $K$ are embedded
into $\End_K(U)$ as scalar matrices with respect to some basis. The more
general case treated in Theorem \ref{quasiReg} corresponds to an embedding of
$K$ into $\End_K(U)$ as diagonal matrices, or, in other words, an embedding
where all linear mappings $\rho_k$, $k\in K$, are simultaneously
diagonalizable.

\section{Chains}

Now  we  consider projective representations of generalized chain geometries.
Let $F$ be a subfield of a ring $R$. Then the  projective line $\PP(F)$ is
embedded into $\PP(R)$ via $F(x,y)\mapsto R(x,y)$. We call the subset $\PP(F)$
the {\em standard chain} of $\PP(R)$. Its images under $\GL_2(R)$ are the {\em
chains}, the set of all chains is denoted by $\fC(F,R)$.
 The incidence structure  $\Sigma(F,R)
=(\PP(R),\fC(F,R))$ is  the {\em generalized chain geometry} over $(F,R)$ as
investigated in \cite{blu+h-99a}. Note that some basic properties of such
geometries have already been derived in  \cite{hot-76}, \cite{bart-89}. Recall
that two distinct points of $\PP(R)$ are {\em distant}, if there is a
$\gamma\in\GL_2(R)$ mapping the first point to $R(1,0)$ and the second point to
$R(0,1)$, or, equivalently, if they are joined by a chain.

Let $\Phi$ be a projective representation  of $\PP(R)$, associated to  a
$(K,R)$-bimodule $U$. We want to determine the $\Phi$-images of the chains of
$\Sigma(F,R)$.  Since $F\subset R$, we have that $U$ is at the same time a
$(K,F)$-bimodule, faithful if  $U\ne \{0\}$, which in turn gives rise to the
projective representation $\Phi|_{\PP(F)}$ of $\PP(F)$. By \cite{blu+h-00a},
Remark 4.1, the $\Phi$-images of any two chains of $\Sigma(F,R)$ are
projectively equivalent. So it suffices to study $\PP(F)^\Phi$, and we can make
use of the results of the previous sections.

From Theorem \ref{regulus}  we know which projective representations of
$\PP(F)$ map  $\PP(F)$ to a regulus. Essentially, these representations  are
given by embeddings of $F$ into some $\End_K(U)$, with $K\cong F$, via a basis,
as described in (\ref{lambda1}), (\ref{lambda2}). In particular, every left
vector space over $K$ can be turned into a suitable $(K,F)$-bimodule, if
$K\cong F$. Now, in addition,  $F$ is a subfield of  $R$, and we are looking
for $(K,R)$-bimodules such that the induced $(K,F)$-bimodule gives rise to
reguli.   Under certain conditions on $F\subset R$ no projective representation
of $\Sigma(F,R)$ maps chains to reguli:

\begin{exa}\label{noRep}
Let $F$ be commutative, and assume that the multiplicative group $F^*$ is not
normal in $R^*$. Consider a projective representation $\Phi$ of $\Sigma(F,R)$
into some non-empty projective space $\PP(K,U\times U)$. If $K\not\cong F$ then
the $\Phi$-images of the  chains do not have any transversals and hence cannot
be reguli. So let $K\cong F$. Then  any three pairwise complementary subspaces
of
 $\PP(K,U\times U)$ lie together in exactly one
regulus (compare \cite{grund-81}, Proposition). By \cite{blu+h-99a}, Theorem
2.4, any three pairwise distant points of $\Sigma(F,R)$ are joined by more than
one chain, which again implies that the $\Phi$-images of the  chains cannot be
reguli.

A class of  examples for this is described in \cite{blu+h-99a}, Example 2.8:
Let $K=\GF(q)$, $F=\GF(q^2)$,  $R=\M(2\times 2,K)$, and $q\ne2$. Then $F^*$ is
not normal in $R^*$. Note that according to \cite{herz-95}, Example 4.5.(2), in
this case the images of the chains under the projective representation
associated to the $(K,R)$-bimodule $K^2$ are regular spreads.

Another class of examples are the geometries $\Sigma(K,L)$, where $L$ is a
quaternion skew field and $K$ is one of its maximal commutative subfields,
studied in \cite{havl-94a}.
\end{exa}

We proceed with some examples where the chains appear as reguli. This can be
checked immediately with the help of Theorem \ref{regulus}. We also give an
explicit geometric description of all the reguli that are images of chains; the
proof of this is left to the reader.

\begin{exas}\label{regExas}
Let $K$ be a not necessarily commutative field. In the following cases the
$\Phi$-images of the chains of $\Sigma(K,R)$  are reguli in $\PP(K,U\times U)$.
\begin{enumerate}
\item Let $R=\End_K(U)$, where $K$ is embedded into $R$ via a basis,
according to (\ref{lambda1}),~(\ref{lambda2}). Then the $\Phi$-images of chains
are all reguli in $\PP(K,U\times U)$ (compare \cite{blunck-00a}, Theorem 3.4).
\item Let $R=K^n$, with componentwise addition and multiplication, let $K$ be
 embedded into $R$ via $k\mapsto (k,\ldots,k)$, and let $U=R$. The $\Phi$-images
 of  chains are exactly  those reguli in $\PP(K,U\times U)$ that have
 the lines $U_i\times U_i$  among their transversals (compare \cite{blu+h-00a},
 Example 5.3).
\item Let $R=K(\eps)=K+K\eps$, the ring of dual numbers
 over $K$, and let $U=R$.  The $\Phi$-images of chains are exactly those reguli in
 $\PP(K,U\times U)$ that have the line $T=K\eps\times K\eps$ as a transversal
 and whose unique element through $p\in T$ lies in the plane
 $p^\beta$ (compare \cite{blu+h-00a}, Example 5.4).
 \item Let $R$ be the ring of upper triangular $2\times 2$-matrices over $K$,
 with $K$ embedded as the scalar matrices. Then  $U=K^2$ is a $(K,R)$-bimodule
 in the natural way. The $\Phi$-images of  chains are exactly those
 reguli that have the line $K(0,1)\times K(0,1)$ among their transversals (compare
 \cite{blu+h-00a},  Example 5.5).

\end{enumerate}
\end{exas}

In all these examples the $\Phi$-images of the chains of $\Sigma(K,R)$ are
exactly the reguli entirely contained in $\PP(R)^\Phi$. This is not  true in
general, as the following example shows. Since in this counterexample $R$  is
infinite-dimensional over $K$, one might conjecture that if $R$ is
finite-dimensional over $K$, and the $\Phi$-images of the chains are reguli,
then all reguli in $\PP(R)^\Phi$ are obtained in this way.

\begin{exa} (Compare \cite{blu+h-00a}, Example 4.7.) Let $R=K[X]$ be the polynomial
ring over a commutative field $K$, and let $U=K(X)$ be its field of fractions.
Then $U$ is in a natural way a faithful $(K,R)$-bimodule and a faithful
$(K,U)$-bimodule. Let $\Phi_1:\PP(R)\to \cG$ and $\Phi_2:\PP(U)\to\cG$ be the
associated faithful representations. Then $\Phi_1=\overline\iota\Phi_2$, where
$\overline\iota:\PP(R)\to\PP(U)$ is induced by the natural inclusion
$\iota:R\to U$ according to \cite{blu+h-00a}, Section 3. One can easily check
that $\overline\iota$ is a bijection and that $p = R(1,0)$ and $q=R(1,X)$ are
non-distant  in $\PP(R)$ but $p^{\overline\iota}$, $q^{\overline\iota}$ are
distant in $\PP(U)$.\\ Now we consider the chain geometries $\Sigma(K,R)$ and
$\Sigma(K,U)$. The images of the chains under $\Phi_1$ and $\Phi_2$ are reguli
contained in $\PP(R)^{\Phi_1}=\PP(U)^{\Phi_2}$.  Recall that  ``distant'' means
``joined by a chain''. So the images $p^{\Phi_1}=p^{\overline\iota\Phi_2}$ and
$q^{\Phi_1}=q^{\overline\iota\Phi_2}$ lie together on a regulus $\cR$ that is
entirely contained in $\PP(R)^{\Phi_1}$, namely, the $\Phi_2$-image of a chain
in $\Sigma(K,U)$ joining $p^{\overline\iota}$ and  $q^{\overline\iota}$.
However, since $p,q\in\PP(R)$ are non-distant, they are not joined by a chain
in $\Sigma(K,R)$, and hence $\cR$ does not appear as ${\Phi_1}$-image of a
chain.
\end{exa}

In case $K=F$, the ring $R$ itself is a faithful $(K,R)$-bimodule with respect
to the regular representation. The associated projective representation  $\Phi$
is the identity.  We apply Theorem~\ref{regulus} to this situation.

\begin{thm}\label{reguliUR}
Let $U=R$ be a ring with $1$ such that $K\subset R$, $1\in K$, and let
$\phi:K\to \End_K(U)$ be given by $u\cdot k:=uk$. Then  $\PP(K)^\Phi=\PP(K)$ is
a regulus if, and only if, there exists a basis $(b_i)_{i\in I}$ of the left
vector space $U=R$ over $K$ such that each $b_i$ centralizes $K$, i.e.,
$b_ik=kb_i$ holds for all $i\in I$, $k\in K$.
\end{thm}
\begin{proof}
Let $\PP(K)^\Phi$ be a regulus. By Theorem \ref{regulus}(c) there is a basis
$(u_i)_{i\in I}$ of $U=R$ and an automorphism $\alpha\in\Aut(K)$ such that $u_i
k=k^\alpha u_i$ holds for all $k\in K$, $i\in I$. This means that the lines
$T_i=Ku_i\times Ku_i$ are transversals. By \ref{exasTrsv}(d), also $T=K\times
K$ is a transversal, belonging to the eigenvector $1\in R$ and the automorphism
$\id_K$. Since any two transversals of $\PP(K)^\Phi$ are projectively linked by
\ref{regulus}(c), this means that $\alpha$ is an inner automorphism. By an
appropriate change of  basis as in (\ref{formel}), the automorphism $\alpha$
goes over to the identity, as desired. The converse follows directly from
Theorem \ref{regulus}.
\end{proof}

A special case is the one where $K$ belongs to the centre $Z(R)$ of $R$, i.e.,
$R$ is a $K$-algebra. This is the case of ordinary chain geometries.

We consider a wider class of examples:
\begin{exa}
Let $K$ be a field and let $Z$ be the centre of $K$. Let $C$ be any
$Z$-algebra, with basis $(c_i)_{i\in I}$. Then the tensor product $R=K\otimes_Z
C$ is a $Z$-algebra  containing $K\cong K\otimes 1$. By \cite{lam-91}, Theorem
(15.1), the centralizer $Z_R(K)$ of $K$ in $R$ is $C\cong 1\otimes C$, and
$(1\otimes c_i)_{i\in I}$ is a $K$-basis of $R$ centralizing~$K$. So in this
situation, the chains of $\Sigma(K,R)$ are mapped to reguli by the projective
representation associated to the $(K,R)$-bimodule $U=R$.

Under certain conditions, also the converse  holds (cf.\ \cite{jac-89}, Theorem
4.7): Let $R$ be a ring, let $K$ a subfield of $R$, and let $Z$ be the centre
of $K$. Assume that $R$ possesses a $K$-basis consisting of elements of  the
centralizer $C:=Z_R(K)$. Then $R$ is a $Z$-algebra and $C$ is a $Z$-subalgebra
of $R$. If in addition $K$ is finite-dimensional over $Z$, then $R\cong
K\otimes_Z C$.
\end{exa}

This means that  we know all pairs $(K,R)$, with $K$  a subfield of $R$ that is
finite-dimensional over its centre, where the projective representation
associated to $U=R$ maps  chains to reguli.

\section{An application}

\newcommand{\oT}{{\omega\T}}
In this final section we show how  projective models of  generalized chain
geometries can be used in order to determine all isomorphisms for a certain
class of such geometries.

Throughout this section $K$ and $K'$ are fields. The rings of $2\times 2
$-matrices over $K$ and $K'$ are denoted by $R$ and $R'$, respectively. The
unit matrices are written as $E$ and $E'$. Each isomorphism $\kappa:K\to K'$
determines an isomorphism $R\to R'$ by putting $(c_{ij})\mapsto
(c_{ij}^\kappa)=:(c_{ij})^\kappa$. Similarly, each antiautomorphism $\omega$ of
$K$ gives rise to an antiautomorphism of $R$ with $(c_{ij})\mapsto
(c_{ji}^\omega)=:(c_{ij})^\oT$.

First we show the following preliminary result on mappings preserving
the distant-relation $\dis$ (as defined before Remark 2.6 in
\cite{blu+h-00a}):

\begin{thm}\label{distIso}
Let $R$ and $R'$ be the rings of\/ $2\times 2$-matrices over fields
$K$ and $K'$, respectively. Then the following assertions on a
mapping $\alpha:\PP(R)\to\PP(R')$ are equivalent:
\begin{enumerate}
\item $\alpha$ is a mapping of the form
           \begin{equation}\label{1}
              \PP(R)\to\PP(R') : R(A,B)\mapsto R'((A^\kappa,B^\kappa)\cdot H')
           \end{equation}
           with $\kappa:K\to K'$  an isomorphism and $H'\in\GL_2(R')$; or $\alpha$ is
           the product of a mapping
           \begin{equation}\label{2}
              \PP(R)\to\PP(R) : R(A,B)\mapsto \{(X,Y)\in R^2\mid -XB^\oT
              + YA^\oT =0 \},
           \end{equation}
           where $\omega$ is an antiautomorphism of $K$,
           with a mapping (\ref{1}).
\item $\alpha$ is a {\em distant-preserving
            bijection}, i.e.,
            $p\dis q\Rightarrow p^\alpha\dis q^\alpha$
            for all $p,q\in\PP(R)$.
\item  $\alpha$ is a bijection which is distant-preserving in both directions.
\end{enumerate}
\end{thm}

\begin{proof}
(a) $\Rightarrow$ (c): The mappings (\ref{1}) are induced by semilinear
bijections  $R^2\to R'^2$, whence they are distant-preserving in both
directions.

If we are given an antiautomorphism $\omega$ of $K$ then
\begin{equation}\label{3}
\Mat2{A&B\\C&D} \mapsto
\Mat2{D^\oT&-B^\oT\\-C^\oT&A^\oT}^{-1}
\end{equation}
is an automorphism of $\GL_2(R)$, as it is the product of the
contragredient automorphism $(W_{ij})\mapsto
{(W^\oT_{ji})}^{-1}$ and an inner automorphism.

Next we show that (\ref{2}) is a well-defined bijection of $\PP(R)$. Let
$R(A,B)$ be a point. Then there exists a matrix $\SMat2{A&B\\C&D}\in \GL_2(R)$.
Put $\SMat2{\widetilde A&\widetilde B\\\widetilde C&\widetilde D}
:=\SMat2{D^\oT&-B^\oT\\-C^\oT&A^\oT}^{-1}$.
With the substitution $(X,Y)=(\widetilde X,\widetilde Y) \SMat2{\widetilde
A&\widetilde B\\\widetilde C&\widetilde D}$ the linear equation appearing in
(\ref{2}) can be rewritten as
\begin{displaymath}
  (\widetilde X,\widetilde Y)  \Mat2{\widetilde A&\widetilde B\\\widetilde C&\widetilde D}
  \Mat2{\widetilde A&\widetilde B\\\widetilde C&\widetilde D}^{-1}  {0\choose E}
  = 0.
\end{displaymath}
The solutions of this equation form the submodule $R(E,0)$ of $R^2$, whence the
equation appearing in (\ref{2}) defines a point, namely the point $R(\widetilde
A,\widetilde B)$.

Similarly, the image of $R(C,D)$ under (\ref{2}) is $R(\widetilde C,\widetilde
D)$. From this observation it is immediate that (\ref{2}) defines a bijection
of $\PP(R)$ which is distant-preserving in both directions.

(c) $\Rightarrow$ (b): This is obviously true.

(b) $\Rightarrow$ (a): In a first step we translate the first part of
this proof into a projective model of $\PP(R)$: The ring $R$ acts on
the left vector space $U:=K^2$ in the natural way and thus turns
$K^2$ into a $(K,R)$-bimodule. The isomorphism $\phi:=\id_R$ yields a
faithful projective representation $\Phi:\PP(R)\to\cG$, where $\cG$
is the set of lines of the projective space $\PP(K,K^4)$. Recall that
now
\begin{displaymath}
  R(A,B)^\Phi = \mbox{left rowspace}(A,B).
\end{displaymath}
From \cite{blu+h-00a}, Remark 4.1,  $\PP(R)^\Phi=\cG$. A bijection
$\Phi':\PP(R')\to \cG'$ is defined similarly.

It is easily seen that for each  $\alpha$ given by (\ref{1}) the mapping
$\Phi^{-1}\alpha\Phi'$ is a bijection $\cG\to\cG'$ which is induced by the
collineation $\PP(K,K^4)\to \PP(K',K'^4)$ given by the matrix
$H'\in\GL_2(R')=\GL_4(K')$ and the isomorphism $\kappa:K\to K'$. Likewise, the
$\Phi$-transform of a mapping (\ref{2}) is a bijection of $\cG$ which is
induced by the correlation of $\PP(K,K^4)$ given by the matrix
$\SMat2{0&-E\\E&0}\in\GL_2(R)=\GL_4(K)$ and the antiautomorphism $\omega$ of
$K$. The linear equation (\ref{2}) now simply means the following: If a line is
given by two of its points then its correlative image arises as the
intersection of two planes.

Obviously, the  mappings (\ref{1}) yield all collineations $\PP(K,K^4)\to
\PP(K',K'^4)$, whence also all correlations $\PP(K,K^4)\to \PP(K',K'^4)$ are
obtained via  all products of a mapping (\ref{2}) with a mapping (\ref{1}).

Now let $\alpha$ be given according to (b). We read off from \cite{blu+h-00a},
Remark 4.1,  that the bijection $\Phi^{-1}\alpha\Phi':\cG\to\cG'$ maps skew
lines to skew lines. By \cite{brau-88}, Satz 2, and \cite{havl-95}, Theorem 3,
the inverse of this bijection of lines is induced either by a collineation or a
correlation $\PP(K',K'^4)\to \PP(K,K^4)$. Hence the assertion follows.
\end{proof}

Note there for the special case $R=R'$ one could also use \cite{huang-98},
Theorem 1,  in order to prove that the line mapping under consideration is
induced by a collineation or correlation. The statement of \cite{huang-98},
Theorem 1, is also valid in higher dimensions, but  unfortunately it is not
applicable to the situation of distant-preserving permutations of the
projective line over a ring of $n\times n$-matrices.

\begin{rem}
Since $R$ is a ring of stable rank $2$ (cf.\ \cite{veld-85}, 2.6), the points
of $\PP(R)$ are exactly the submodules of the form $R(A,E+AB)$ with $A,B\in R$
(cf.\ \cite{bart-89}). Hence the mapping (\ref{2}) can be written in the
explicit form
\begin{equation}\label{4}
  R(A,E+AB)\mapsto R(A^\oT,E+A^\oT B^\oT).
\end{equation}
Cf. also  \cite{herz-95}, Theorem 9.1.1, and \cite{bart-89}, Theorem 2.4.
\end{rem}

Now let $F\subset R$ and $F'\subset R'$ be fields with $E\in F$, $E'\in F'$.
Recall that $\GL_2(R')$ operates transitively on the set of chains of
$\Sigma(F',R')$. Moreover, the stabilizer in $\GL_2(R')$ of a chain acts
$3$-transitively on its set of points (see \cite{blu+h-99a}, Theorem 2.3).
Thus, if we want to find all isomorphisms of the generalized chain geometry
$\Sigma(F,R)$ onto $\Sigma(F',R')$, then it is sufficient to determine all
fundamental isomorphisms; here a mapping is called {\em fundamental} if it
takes the standard chain of $\Sigma(F,R)$ into the standard chain of
$\Sigma(F',R')$ and the points $R(E,0)$, $R(0,E)$, $R(E,E)$ to $R'(E',0')$,
$R'(0',E')$, $R'(E',E')$, respectively.

In the next theorem we consider the more general case of bijective morphisms,
where a {\em morphism} maps chains   {\em into} chains. As above, also in this
situation it suffices to study the fundamental ones.

\begin{thm}\label{auto}
Let $\Sigma(F,R)$ and $\Sigma(F',R')$ be generalized chain geometries, where
$R$ and $R'$ are rings of $2\times 2$-matrices over fields $K$ and $K'$,
respectively. Then the fundamental bijective morphisms
$\Sigma(F,R)\to\Sigma(F',R')$ are exactly the following mappings
$\PP(R)\to\PP(R')$:
           \begin{equation}\label{5}
              R(A,B)\mapsto R'\left((A^\kappa,B^\kappa)\Mat2{H'_1&0'\\0'&H'_1}\right)
           \end{equation}
           where $\kappa:K\to K'$ is an isomorphism, $H'_1\in\GL_2(K')$,
           and ${H'_1}^{-1} F^\kappa H'_1\subset F'$;
           or the product of a mapping
           \begin{equation}\label{6}
            R(A,E+AB)\mapsto R(A^\oT,E+A^\oT B^\oT),
           \end{equation}
           where $\omega$ is an antiautomorphism of $K$,
            with a mapping (\ref{5}), where
           $\kappa:K\to K'$ is an isomorphism and  $H'_1\in\GL_2(K')$
            such that
           ${H'_1}^{-1} (F^\oT)^\kappa H'_1\subset F'$.

\end{thm}

\begin{proof}
Each fundamental bijective morphism $\alpha$ is a distant-preserving bijection.
Thus we can apply Theorem \ref{distIso}. There are two cases:

Let $\alpha$ be given according to (\ref{1}). Then $H'$ has necessarily the
form $\SMat2{H'_1&0'\\0'&H'_1}$ with $H'_1\in\GL_2(K')$, since $R(E,0)$,
$R(0,E)$, $R(E,E)$ go over to $R'(E',0')$, $R'(0',E')$, $R'(E',E')$. Any point
of the standard chain other than $R(E,0)$ has the form $R(M,E)$, where $M\in
F\subset R$. Its image point is $R'({H'_1}^{-1}M^\kappa H'_1,E')$, so that
${H'_1}^{-1}F^\kappa H'_1\subset F'$, as required.

Let $\alpha$ be a product of a mapping (\ref{2}) and a mapping
(\ref{1}). Observe that each mapping (\ref{2}) fixes $R(E,0)$,
$R(0,E)$, and $R(E,E)$. Thus the matrix $H'$ appearing in
(\ref{1}) has the form as in the previous case. Now a similar
argument as before together with (\ref{4}) the yields assertion.

For the proof of the converse we consider the following commutative diagram:
\begin{diagram}[small]
\PP(R)& & &\rTo^\alpha& & &\PP(R')\\
 &\rdTo^\Phi& & & &\ldTo^{\Phi'}& \\
 & &\cG&\rTo&\cG'& & \\
\dTo^\gamma&&\dTo& &\dTo& &\dTo_{\gamma'}\\
 & &\cG&\rTo&\cG'& & \\
 &\ruTo^\Phi& & & &\luTo^{\Phi'}& \\
\PP(R)& & &\rTo^\alpha& & &\PP(R')
\end{diagram}

Let $\alpha:\PP(R)\to\PP(R')$ be as in the assertion, and let  $\gamma$ be a
projectivity of $\PP(R)$ given by a matrix in $\GL_2(R)$. The same matrix,
considered as an element of $\GL_4(K)$, yields a projective collineation $\pi$
of $\PP(K,K^4)$, which gives rise to the permutation $\Phi^{-1}\gamma\Phi$ of
the line set $\cG$. Note that this does not depend on the choice of the matrix
because the kernels of the actions of $\GL_2(R)$ on $\PP(R)$ and on $\cG$
coincide as the centre of $K$ equals the centre of $R$. Moreover, the proof of
Theorem \ref{distIso} shows that $\Phi^{-1}\alpha\Phi':\cG\to\cG'$ is induced
by a collineation or correlation $\mu:\PP(K,K^4)\to\PP(K',K'^4)$. So
$\mu^{-1}\pi\mu$ is a projective collineation of $\PP(K',K'^4)$, whence it can
be described by a matrix in $\GL_4(K')=\GL_2(R')$. As before, this matrix
yields the projectivity $\gamma':=\alpha^{-1}\gamma\alpha$ of $\PP(R')$.

Now the assertion easily follows, because obviously $\alpha$ maps the standard
chain $\PP(F)$ of  $\Sigma(F,R)$ into the standard chain of $\Sigma(F',R')$,
and any other chain of $\Sigma(F,R)$  has the form $\PP(F)^\gamma$ for some
$\gamma$ as in the diagram.
\end{proof}

We want to mention here that the second direction of this theorem also follows
from the more general statement in \cite{bart-89}, Theorem (2.4).

Of course Theorem \ref{auto} also yields an algebraic description of all
fundamental isomorphisms. Then the conditions on the image of $F$ read
${H'_1}^{-1}F^\kappa H_1'=F'$ and ${H'_1}^{-1}(F^\oT)^\kappa H_1'=F'$,
respectively.

\begin{thebibliography}{10}

\bibitem{bart-89}
C.G. Bartolone.
\newblock Jordan homomorphisms, chain geometries and the fundamental theorem.
\newblock {\em Abh. Math. Sem. Univ. Hamburg}, 59:93--99, 1989.

\bibitem{benz-73}
W.~Benz.
\newblock {\em Vorlesungen \"uber Geometrie der Algebren}.
\newblock Springer, Berlin, 1973.

\bibitem{blunck-00a}
A.~Blunck.
\newblock Reguli and chains over skew fields.
\newblock {\em Beitr\"age Algebra Geom.}, 41:7--21, 2000.

\bibitem{blu+h-99a}
A.~Blunck and H.~Havlicek.
\newblock Extending the concept of chain geometry.
\newblock {\em Geom. Dedicata}
\newblock (to appear).

\bibitem{blu+h-00a}
A.~Blunck and H.~Havlicek.
\newblock Projective representations {I}. {P}rojective lines over rings.
\newblock submitted to {\em Abh. Math. Sem. Univ. Hamburg}.

\bibitem{brau-88}
H.~Brauner.
\newblock {\"U}ber die von {K}ollineationen projektiver {R}\"aume induzierten
  {G}eradenabbildungen.
\newblock {\em Sb. \"osterr. Akad. Wiss, Abt. II, math. phys. techn. Wiss.},
  197:327--332, 1988.

\bibitem{grund-81}
T.~Grundh\"ofer.
\newblock Reguli in {F}aserungen projektiver {R}\"aume.
\newblock {\em Geom. Dedicata}, 11:227--237, 1981.

\bibitem{havl-94a}
H.~Havlicek.
\newblock Spreads of right quadratic skew field extensions.
\newblock {\em Geom. Dedicata}, 49:239--251, 1994.

\bibitem{havl-95}
H.~Havlicek.
\newblock On isomorphisms of {G}rassmann spaces.
\newblock {\em Mitt. Math. Ges. Hamburg}, 14:117--120, 1995.

\bibitem{herz-95}
A.~Herzer.
\newblock Chain geometries.
\newblock In F.~Buekenhout, editor, {\em Handbook of Incidence Geometry}.
  Elsevier, Amsterdam, 1995.

\bibitem{hot-76}
H.~Hotje.
\newblock Zur {E}inbettung von {K}ettengeometrien in projektive {R}\"aume.
\newblock {\em Math. Z.}, 151:5--17, 1976.

\bibitem{huang-98}
W.-l. Huang.
\newblock Adjacency preserving transformations of {G}rassmann spaces.
\newblock {\em Abh. Math. Sem. Univ. Hamburg}, 68:65--77, 1998.

\bibitem{jac-89}
N.~Jacobson.
\newblock {\em Basic Algebra II}.
\newblock Freeman, New York, 1989.

\bibitem{lam-91}
T.Y. Lam.
\newblock {\em A First Course in Noncommutative Rings}.
\newblock Springer, New York, 1991.

\bibitem{veld-85}
F.D. Veldkamp.
\newblock Projective ring planes and their homomorphisms.
\newblock In R.~Kaya, P.~Plaumann, and K.~Strambach, editors, {\em Rings and
  Geometry}. D.\ Reidel, Dordrecht, 1985.

\end{thebibliography}

\bigskip

Institut f\"ur Geometrie\\
Technische Universit\"at\\
Wiedner Hauptstra{\ss}e 8--10\\
A--1040 Wien,
Austria
\end{document}